\documentclass[12pt, oneside, psamsfonts]{amsart}

\newif\ifPDF
\ifx\pdfoutput\undefined\PDFfalse
\else \ifnum \pdfoutput > 0 \PDFtrue
        \else \PDFfalse
        \fi
\fi

\usepackage[centertags]{amsmath}
\usepackage{amsfonts}
\usepackage{mathrsfs}
\usepackage{textcomp}
\usepackage{amssymb}
\usepackage{amsthm}
\usepackage{newlfont}
\usepackage[all]{xy}

\ifPDF
  \usepackage[pdftex]{color, graphicx}
  \usepackage[pdftex, bookmarks, colorlinks]{hyperref}
  \hypersetup{colorlinks=false}

\else
  \usepackage{color}
  \usepackage[dvips]{graphicx}
  \usepackage[dvips]{hyperref}
\fi

\usepackage[scale=0.8]{geometry}

\usepackage[pagewise, mathlines, displaymath]{lineno}

\newtheorem{thm}{Theorem}[section]
\newtheorem{cor}[thm]{Corollary}
\newtheorem{lem}[thm]{Lemma}
\newtheorem{prop}[thm]{Proposition}

\theoremstyle{definition}
\newtheorem{defn}[thm]{Definition}
\theoremstyle{remark}
\newtheorem{rem}[thm]{Remark}
\newtheorem{example}[thm]{Example}
\numberwithin{equation}{section}

\newcommand{\norm}[1]{\left\Vert#1\right\Vert}
\newcommand{\abs}[1]{\left\vert#1\right\vert}

\newcommand{\Real}{\mathbb R}
\newcommand{\Int}{\mathbb Z}
\newcommand{\Comp}{\mathbb C}

\newcommand{\eps}{\varepsilon}

\newcommand{\tr}{\mathrm{T}}

\begin{document}


\title{On the small boundary property, $\mathcal Z$-absorption, and Bauer simplexes}

\author{George A.~Elliott}
\address{Department of Mathematics, University of Toronto, Toronto, Ontario, Canada~\ M5S 2E4}
\email{elliott@math.toronto.edu}

\author{Zhuang Niu}
\address{Department of Mathematics and Statistics, University of Wyoming, Laramie, Wyoming 82071, USA}
\email{zniu@uwyo.edu}



\begin{abstract}
Let $X$ be a metrizable compact space, and let $\Delta$ be a closed set of Borel probability measures on $X$. We study the small boundary property of the pair $(X, \Delta)$. In particular, it is shown that $(X, \Delta)$ has the small boundary property if it has a certain divisibility property. 

As an application, it is shown that, if $A$ is the crossed product C*-algebra $\mathrm{C}(X)\rtimes\mathbb Z^d$, where $(X, \mathbb Z^d)$ is a free minimal topological dynamical system, or if $A$ is an AH algebra with diagonal maps, then  $A$ is $\mathcal Z$-absorbing if the set of extreme tracial states is compact, regardless of its dimension.

\end{abstract}

\maketitle


\section{Introduction}

The small boundary property was introduced in \cite{Lindenstrauss-Weiss-MD} as a dynamical systems  analogue of the usual definition of zero dimensional space. It was shown to be equivalent to zero mean dimension (\cite{Lind-MD}; see \cite{GLT-Zk} for $\Int^d$-actions), and implies the $\mathcal Z$-stability of the crossed product C*-algebra $\mathrm{C}(X)\rtimes\Int$ (\cite{EN-MD0}). In this paper, we shall formulate the small boundary property for a pair $(X, \Delta)$ (see Definition \ref{defn-sbp}), where $X$ is a compact metrizable space and $\Delta$ is a closed set of Borel probability measures on $X$. This clearly includes the dynamical system case where $\Delta$ consists of the invariant probability measures, but also includes the examples where $\mathrm{C}(X)$ is a Cartan subalgebra of $A$ and $\Delta$ is the trace simplex of $A$. 

As suggested by the arguments of \cite{Lind-MD} and \cite{GLT-Zk},  the small boundary property can be characterized in the context of the uniform trace norm:
\begin{thm}[Theorem \ref{SBP-2-norm}]
 $(X, \Delta)$ has the (SBP) if, and only if, for any continuous real-valued function $f: X \to \Real$ and any $\eps>0$, there is a continuous real-valued function $g: X\to \Real$ such that
\begin{enumerate}
\item $\norm{f - g}_{2, \Delta} < \eps$, and 
\item $\mu(g^{-1}(0)) < \eps$ for all $\mu \in \Delta$. 
\end{enumerate} 
\end{thm}
This property of $\mathrm{C}(X)$ (with respect to $\Delta$) is similar to the property of real rank zero of $\mathrm{C}(X)$, which, for an arbitrary C*-algebra, asserts that any self-adjoint element can be approximated by invertible self-adjoint elements. Indeed, the (SBP) can be characterized by zero real rank of the sequence algebra of $\mathrm{C}(X)$ modulo the trace kernel of $\Delta$ (Theorem \ref{RR0=SBP}; see also Theorem \ref{RR0=SBP-X}).

In order to apply this theorem, consider the following property, reminiscent of uniform property $\Gamma$ (Definition 2.1 of \cite{CETW-Gamma}). 
\begin{defn}[Definition \ref{Definition-WGamma}]
The pair $(X, \Delta)$ will be said to be approximately divisible if there is $K>0$ such that for each  $n \in \mathbb N$, there is a partition of unity (mutually orthogonal projections with sum $1$) $$p_1, p_2, ..., p_n \in \ell^\infty(D)/J_{2, \omega, \Delta}$$ such that
$$\tau(p_iap_i) \leq \frac{1}{n} K \tau(a),\quad a\in D^+,\ \tau\in \Delta_\omega,$$
where $D = \mathrm{C}(X)$.
\end{defn}
This property turns out to imply the (SBP): 
\begin{thm}[Theorem \ref{Gamma2SBP}]\label{thm13-int}
If $(X, \Delta)$ is approximately divisible, then $(X, \Delta)$ has the (SBP).
\end{thm}

In the case that $\mathrm{C}(X)$ is the canonical subalgebra of $A=\mathrm{C}(X)\rtimes\Int^d$, where $\Int^d$ acts minimally and freely on $X$, or in the case that  $\mathrm{C}(X)$ is the diagonal subalgebra of $A$, where $A$ is an AH algebra with diagonal maps, if $\Delta = \mathrm{T}(A)$ is a Bauer simplex (i.e., $\partial \Delta$ is compact), then $(X, \Delta)$ is approximately divisible (see Lemma \ref{Bauer-case} below). Therefore by Theorem \ref{thm13-int}, it must have the (SBP). As we shall show, the C*-algebra $A$ in these cases must be $\mathcal Z$-absorbing  (Theorem \ref{Bauer-SBP} and Proposition \ref{AH-Z}). Compared to the literature where Bauer simplexes are used in the classification of C*-algebras (for instance, \cite{KR-CenSeq}, \cite{TWW-Z}, and \cite{Sato-CP}), we do not use W*-bundles, and we make no assumption on the topological dimension of $\partial \Delta$.

\begin{thm}[cf.~Theorems \ref{Bauer-SBP} and \ref{Bauer-SBP-AH}]
Let $A$ be a C*-algebra as above. If $\mathrm{T}(A)$ is a Bauer simplex, then $(X, \Delta)$ has the (SBP) and $A \cong A\otimes\mathcal Z$, where $\mathcal Z$ is the Jiang-Su algebra.
\end{thm}

\subsection*{Acknowledgements}
The results of this note were reported in the Operator Algebra Seminar at the Fields Institute in July, 2023 (\cite{Fields-OAS}). The authors thank the Fields Institute for support. The research of the first named author was supported by a Natural Sciences and Engineering Research Council of Canada (NSERC) Discovery Grant, and the research of the second named author was supported by a U.S.~National Science Foundation grant (DMS-1800882) and a Simons Foundation grant (MP-TSM-00002606).

\section{Small boundary property}

\subsection{Small boundary property}
%

\begin{lem}\label{SBP-conds}
Let $X$ be a compact metrizable space, and let $\Delta$ be a set of Borel probability measures.  The following properties are equivalent:
\begin{enumerate}
\item for any subsets $V \subseteq U$, where $V$ is closed and $U$ is open, there is a closed set $V'$ such that $V \subseteq V'\subseteq U$, and $\mu(\partial V') = 0$, for all $\mu\in\Delta$.

\item for any subsets $V \subseteq U$, where $V$ is closed and $U$ is open, there is an open set $V'$ such that $V \subseteq V'\subseteq U$, and $\mu(\partial V') = 0$, for all $\mu\in\Delta$.

\item for any $x \in X$ and any open set $U \ni x$, there is a closed neighbourhood $V \subseteq U$ of $x$ such that $\mu(\partial V) = 0$, for all $\mu\in\Delta$.

\item for any $x \in X$ and any open set $U \ni x$, there is an open neighbourhood $V \subseteq U$ of $x$ such that $\mu(\partial V) = 0$, for all $\mu\in\Delta$.

\end{enumerate}
\end{lem}

\begin{proof}
Since $\overline{\mathrm{int}(W)}\setminus \mathrm{int}(W) \subseteq \partial W=W \setminus \mathrm{int}(W)$ for any closed set $W$, (1)$\Rightarrow$(2) and (3)$\Rightarrow$(4) are immediate. 

For (2)$\Rightarrow$(3), choose an open set $U'$ with $\overline{U'} \subseteq U$. Since $\{x\}$ is closed, by $(2)$, there is an open set $V \ni x$ such that $V \subseteq U'$ and $\mu(\partial V) = 0$ for all $\mu \in \Delta$. Then the closed neighbourhood  $\overline{V}$ of $x$ satisfies (3).

For (4)$\Rightarrow$(1), choose an open set $U'$ with $V \subseteq U'$ and $\overline{U'} \subseteq U$. Since $V$ is compact, by (4), there is an open cover $\{V_1, V_2, ..., V_n\}$ of $V$ such that $V_i\subseteq U'$, $i=1, ..., n$ and $\mu(\partial V_i) = 0$, $i=1, ..., n$, and $\mu\in\Delta$. Then $V':=\overline{V_1} \cup\cdots \cup \overline{V_n}$ satisfies (1).
\end{proof}

\begin{defn}\label{defn-sbp}
Let $X$ be a metrizable compact space, fixed for the rest of the paper, and let $\Delta$ be a closed set of Borel probability measures on $X$. Then the pair $(X, \Delta)$ (or the pair $(\mathrm{C}(X), \Delta)$, where $\Delta$ is regarded as as closed set of tracial state of $\mathrm{C}(X)$) is said to have the Small Boundary Property (SBP) if the equivalent conditions  of Lemma \ref{SBP-conds} are satisfied.
\end{defn}

\begin{lem}\label{lem-small-nbhd}
Let $\Delta$ be a closed set of Borel probability measures on $X$. Let $c\geq 0$, and let $V\subseteq X$ be a closed set with $\mu(V) < c$ for all $\mu\in\Delta$. Then there is an open set $U\supseteq V$ such that $\mu(U) < c$ for all $\mu\in\Delta$. 
\end{lem}

\begin{proof}
Assume the statement were not true. Then there would exist a decreasing sequence  of open sets $(U_i)$ with $\overline{U_{i+1}} \subseteq U_i$ and $\bigcap_{i=1}^\infty U_i = V$ and a sequence $(\mu_i)$ in $\Delta$ such that 
$\mu_i(U_i) \geq c $ for all $i=1, 2, ...$ (and hence $\mu_k(U_i) \geq c$ for all $k\geq i$, $k$, $i=1, 2, ...$). Since $\Delta$ is closed (and so compact), passing to a subsequence, one may assume that $(\mu_i)$ converges to a measure $\mu_\infty\in\Delta$. Then, since $$\mu_\infty(U_i) \geq \mu_\infty(\overline{U_{i+1}})\geq \limsup_{k\to\infty}\mu_k(\overline{U_{i+1}}) \geq \limsup_{k\to\infty}\mu_k(U_{i+1})  \geq c,$$
and $\mu_\infty \in \Delta$, one has $$c> \mu_\infty(V) = \mu_\infty(\bigcap_{i=1}^\infty U_i) = \lim_{i\to\infty} \mu_\infty(U_i)\geq c,$$ which is absurd.
\end{proof}

The following lemma is Proposition 5.3 of \cite{Lindenstrauss-Weiss-MD}, which was stated for dynamical systems in terms of orbit capacity. For the reader's convenience,  we include the proof here. 
\begin{lem}[Proposition 5.3 of \cite{Lindenstrauss-Weiss-MD}]\label{unity-sbp}
If $(X, \Delta)$ has the (SBP), then, for every open cover $\alpha$ of $X$ and every $\eps>0$, there is a subordinate partition of unity $\phi_i: X \to [0, 1]$, $i=1, 2, ..., \abs{\alpha}$, such that
\begin{enumerate}
\item $\sum_{i=1}^{\abs{\alpha}} \phi_i(x) = 1$, $x\in X$,
\item $\mathrm{supp}(\phi_i) \subseteq U$ for some $U\in \alpha$, and $i=1, 2, ..., \abs{\alpha}$, and
\item $\mu(\bigcup_{i=1}^{\abs{\alpha}}\phi_i^{-1}(0, 1))<\eps$ for all $\mu\in\Delta$.
\end{enumerate}
\end{lem}
\begin{proof}
List $\alpha = \{U_1, U_2, ..., U_{\abs{\alpha}}\}$.
Using the (SBP), for each $i = 1, ..., \abs{\alpha}$, one finds an open set $U'_i\subseteq U_i$ such that $U'_i$, $i=1, 2, ..., \abs{\alpha}$, still form an open cover of $X$, and $\mu(\partial U'_i) = 0$ for all $\mu\in\Delta$. By Lemma \ref{lem-small-nbhd}, there is $\delta>0$ such that
$$\mu((\partial U'_i)_\delta) < \eps/\abs{\alpha},$$
where $(\partial U'_i)_\delta$ is the $\delta$-neighbourhood of $\partial U'$. One may assume that $\delta$ is small enough that $(\partial U'_i)_\delta \subseteq U_i$ for each $i=1, ..., \abs{\alpha}$.

Define
$$ \psi_i = \left\{ \begin{array}{ll} 
1, & x\in U'_i, \\
\max(0, 1-\delta^{-1}\mathrm{dist}(x, \partial U'_i)), & \textrm{otherwise}.
\end{array}\right.$$
Then, define
$$
\left\{
\begin{array}{rcl}
\phi_1(x) & = & \psi_1(x), \\
\phi_2(x) & = & \min(\psi_2, 1-\phi_1(x)), \\
\phi_3(x) & = & \min(\psi_3, 1 - \phi_1(x) - \phi_2(x)), \\
& \vdots & \\
\phi_{\abs{\alpha}}(x) & = & \min(\psi_{\abs{\alpha}}, 1 - \phi_1 - \cdots - \phi_{\abs{\alpha} - 1}(x)).
\end{array}
\right.
$$
Then, $\mathrm{supp}(\phi_i) \subseteq U_i$, $i=1, ..., \abs{\alpha}$, and 
$$\bigcup_{i=1}^{\abs{\alpha}} \phi_i^{-1}(0, 1) \subseteq \bigcup_{i=1}^{\abs{\alpha}}\psi_i^{-1}(0, 1).$$
\end{proof}

The following lemma is contained in the proof of Theorem 6.2 of \cite{Lind-MD} and also in the proof of Corollary 5.4 of \cite{GLT-Zk}. We state and prove it for the convenience of the reader.

\begin{lem}[Section 6 of \cite{Lind-MD}; Section 5 of \cite{GLT-Zk}]\label{pre-rr0}
$(X, \Delta)$ has the (SBP) if, and only if, for any continuous real-valued function $f: X \to \Real$ and any $\eps>0$, there is a continuous real-valued function $g: X\to \Real$ such that
\begin{enumerate}
\item $\norm{f - g}_\infty < \eps$,
\item $\mu(g^{-1}(0)) < \eps$ for all $\mu \in \Delta$.
\end{enumerate} 
\end{lem}

\begin{proof}
Assume the properties of the lemma. In order to show that $(X, \Delta)$ has the (SBP), it is enough to show that the set $$\{f \in\mathrm{C}_{\Real}(X): \mu(f^{-1}(0)) = 0,\  \mu\in\Delta \}$$ is dense in $\mathrm{C}_\Real(X)$ (it is always a $G_\delta$ set; see Lemma 6.4 of \cite{Lind-MD}). Indeed, let $x \in X$ and let $U$ be an open neighbourhood of $x$, and pick a continuous function $f: X \to [-1, 1]$ such that $f(x) = -1$ and $f(X\setminus U) = \{1\}$. Since the set above is dense, there is $g$ such that $\norm{g - f} < 1/4$ and $\mu(g^{-1}(0)) = 0$ for all $\mu \in\Delta$. Consider the set $V=g^{-1}([-1, 0]) \subseteq U$. It is a closed neighbourhood of $x$ with $\partial V = g^{-1}(0)$, and hence $\mu(\partial V) = 0$ for all $\mu\in\Delta$. So $(X, \Delta)$ has the (SBP).

For each $c>0$, consider the set
$$Z_c:=\{f \in\mathrm{C}_{\Real}(X): \mu(f^{-1}(0)) < c,\ \mu\in\Delta \}.$$
Note that the properties of the lemma imply that $Z_{c}$ is dense in $\mathrm{C}_\Real(X)$.

Let us show that $Z_c$ is also open in $\mathrm{C}_\Real(X)$. Indeed, let $f: X \to \Real$ be a continuous function such that $\mu(f^{-1}(0)) < c$ for all $\mu\in\Delta$, and consider the closed set $V:=f^{-1}(0)$. By Lemma \ref{lem-small-nbhd}, there is an open set $U\supseteq V$ such that $\mu(U) < c$ for all $\mu\in\Delta$. 

Therefore, the set 
\begin{eqnarray*}
&& \{f \in\mathrm{C}_{\Real}(X): \mu(f^{-1}(0)) = 0,\ \mu\in\Delta \} \\
&=& \bigcap_{n=1}^\infty \{f \in\mathrm{C}_{\Real}(X): \mu(f^{-1}(0)) < 1/n,\  \mu\in\Delta \} \\
& = & \bigcap_{n=1}^\infty Z_{1/n} 
\end{eqnarray*}
is dense in $\mathrm{C}_\Real(X)$ (a dense $G_\delta$, in fact), and $(X, \Delta)$ has the (SBP).

For the converse, assume that $(X, \Delta)$ has the (SBP), and let $f: X \to \Real$ and any $\eps>0$ be given. By Lemma \ref{unity-sbp}, there exists a partition of unity $\phi_i: X \to [0, 1]$, $i=1, 2, ..., N$, such that 
\begin{equation}\label{small-shrink}
\mu(\bigcup_{i=1}^N \phi_i^{-1}((0, 1)) < \eps,\quad \mu\in\Delta,
\end{equation}
and
$$\abs{f(x) - f(y)} < \eps/2,\quad x, y\in \phi_i^{-1}(0, 1]),\ i=1, ..., N.$$
For each $i=1, ..., N$, pick $x_i\in \phi_i^{-1}(0, 1]$ and a real number $y_i$ such that $$y_i\neq 0\quad \textrm{and}\quad \abs{y_i - f(x_i)}<\eps/2.$$
Define
$$g= \sum_{i=1}^N y_i\phi_i.$$
Note that for each $x \in \bigcup_{i=1}^N \phi_i^{-1}(1)$, one has that $g(x) \neq 0$, and hence
$$ g^{-1}(0) \subseteq X \setminus \bigcup_{i=1}^N \phi_i^{-1}(1) =  \bigcup_{i=1}^N \phi_i^{-1}((0, 1)).$$ 
By \eqref{small-shrink}, direct calculations show that
\begin{enumerate}
\item $\norm{f - g}_\infty <\eps$, and
\item $\mu(g^{-1}(0)) < \eps$ for all $\mu\in\Delta$,
\end{enumerate}
as desired.
\end{proof}

\begin{defn}
Let $\Delta$ be a closed subset of probability Borel measures on $X$. Then define
$$\norm{f}_{2, \Delta} = \sup\{(\int \abs{f}^2 d\mu)^{\frac{1}{2}}: \mu\in \Delta\},\quad f\in \mathrm{C}(X).$$
\end{defn}

\begin{lem}[Markov's inequality]
$$ \mu(E_\eps) \leq \frac{1}{\eps^2}\int \abs{f}^2d\mu,$$
where $E_\eps = \{x\in X: \abs{f(x)} > \eps\}$.
\end{lem}
\begin{proof}
$$\eps^2\mu(E_\eps) = \int_{E_\eps} \eps^2 d\mu  \leq \int_{E_\eps} \abs{f}^2 d\mu + \int_{E_\eps^c} \abs{f}^2 d\mu = \int \abs{f}^2 d\mu.$$
\end{proof}

\begin{lem}\label{ta-na}
Let $f, g$ be real-valued continuous functions on $X$ satisfying $$\norm{f - g}_{2, \Delta}^2 < \frac{\eps^3}{8}$$ and $$\mu(g^{-1}(0)) < \frac{\eps}{2},\quad \mu\in\Delta.$$ Then, there is a real-valued continuous function $g'$ on $X$ such that
\begin{enumerate}
\item $\norm{f - g'}_\infty < \eps$, and
\item $\mu((g')^{-1}(0)) < \eps$ for all $\mu\in\Delta$.
\end{enumerate}
\end{lem}

\begin{proof}
Define
$$E_\eps = \{x\in X: \abs{f(x) - g(x)} > \eps/2\}.$$
Then, by Markov's inequality,
$$ \mu(E_\eps) \leq \frac{4}{\eps^2}\norm{f-g}_{2, \Delta}^2 < \frac{\eps}{2},\quad \mu\in\Delta.$$

Pick an open neighbourhood $U$ of $E_\eps^c$ such that 
$$\abs{f(x) - g(x)} < \eps,\quad x\in U,$$
and pick a continuous function $h: X \to [0, 1]$ such that $$h|_{E_\eps^c} = 1 \quad\mathrm{and}\quad h|_{U^c} = 0.$$

Define
$$g'(x) = h(x)g(x) + (1-h(x))f(x).$$
Then 
$$\abs{f(x) - g'(x)} <\eps,\quad x\in X.$$
Moreover, since
$$(g')^{-1}(0) \subseteq g^{-1}(0) \cup E_\eps,$$
one has $$\mu((g')^{-1}(0)) \leq \mu(g^{-1}(0)) + \mu(E_\eps)< \eps,$$
as desired.
\end{proof}

Using this together with Lemma \ref{pre-rr0}, we have the following criterion for the (SBP).
\begin{thm}\label{SBP-2-norm}
 $(X, \Delta)$ has the (SBP) if, and only if, for any continuous real-valued function $f: X \to \Real$ and any $\eps>0$, there is a continuous real-valued function $g: X\to \Real$ such that
\begin{enumerate}
\item $\norm{f - g}_{2, \Delta} < \eps$, and 
\item $\mu(g^{-1}(0)) < \eps$ for all $\mu \in \Delta$; this is equivalent to the existence of $\delta>0$ such that $\tau_\mu(\chi_{\delta}(g)) < \eps$, $\mu \in\Delta$, where 
$$\chi_\delta(t) = 
\left\{
\begin{array}{ll}
1, & |t| < \delta/2, \\
0, & |t| > \delta, \\
\textrm{linear}, & \textrm{otherwise}.
\end{array}
\right.
$$
\end{enumerate} 
\end{thm}

\begin{proof}
By Lemma \ref{pre-rr0} and the fact that $ \norm{a}_{2, \Delta} \leq \norm{a}$, we only need to show that if  $(X, \Delta)$ has the property that for any continuous real-valued function $f: X \to \Real$ and any $\eps>0$, there is a continuous real-valued function $g: X\to \Real$ such that
\begin{enumerate}
\item $\norm{f - g}_{2, \Delta} < \eps$, and 
\item $\mu(g^{-1}(0)) < \eps$ for all $\mu \in \Delta$,
\end{enumerate}
then $(X, \Delta)$ has the (SBP). By Lemma \ref{pre-rr0} again, we only need to show that $(X, \Delta)$ has the property that for any continuous real-valued function $f: X \to \Real$ and any $\eps>0$, there is a continuous real-valued function $g: X\to \Real$ such that
\begin{enumerate}
\item $\norm{f - g}_{\infty} < \eps$, and 
\item $\mu(g^{-1}(0)) < \eps$ for all $\mu \in \Delta$.
\end{enumerate}

With the given $\eps>0$, applying the assumption with $f$ and $\min\{\eps^3/8, \eps/2\}$, there is a continuous real-valued function $g: X\to \Real$ such that
\begin{enumerate}
\item $\norm{f - g}_{2, \Delta} <\eps^3/8$, and 
\item $\mu(g^{-1}(0)) < \eps/2$ for all $\mu \in \Delta$.
\end{enumerate}
Then, by Lemma \ref{ta-na}, there is a real-valued continuous function $g'$ on $X$ such that
\begin{enumerate}
\item $\norm{f - g'}_\infty < \eps$, and
\item $\mu((g')^{-1}(0)) < \eps$ for all $\mu\in\Delta$,
\end{enumerate}
as desired.
\end{proof}

We would like to point out the following corollary, which might be interesting on its own.
\begin{cor}
Consider $(X, \Delta_1)$ and $(Y, \Delta_2)$, and assume there is an embedding $\phi: \mathrm{C}(X) \to \mathrm{C}(Y)$ such that 
\begin{enumerate}
\item $\Delta_1 \subseteq \phi^*(\Delta_2)$,
\item $\phi(\mathrm{C}(X))$ is dense in $\mathrm{C}(Y)$ with respect to $\norm{\cdot}_{2, \Delta_2}$.
\end{enumerate}
If $(Y, \Delta_2)$ has the (SBP), then $(X, \Delta_1)$ also has the (SBP).
\end{cor}

\begin{proof}
Identify $\mathrm{C}(X)$ with a sub-C*-algebra of $\mathrm{C}(Y)$, and identify $\Delta_1$ with a subset of $\Delta_2$.

Let $f: X \to \Real$ be continuous, and let $\eps>0$ be arbitrary. Since $(Y, \Delta_2)$ has the (SBP), by Theorem \ref{SBP-2-norm}, there is $g': Y\to \Real$ such that
\begin{enumerate}
\item $\norm{\phi(f) - g'}_{2, \Delta_2} < \eps$,
\item $\mu((g')^{-1}(0)) < \eps$ for all $\mu \in \Delta_2$.
\end{enumerate} 
Since $\Delta_2$ and $(g')^{-1}(0)$ are closed, by Lemma \ref{lem-small-nbhd}, there is an open neighbourhood of $(g')^{-1}(0)$ with measure uniformly smaller than $\eps$ with respect to $\Delta_2$. It then follows that there is $\delta>0$ such that
$$\tau(\chi_\delta(g')) < \eps,\quad \tau \in \Delta_2.$$

Since $\phi(\mathrm{C}(X))$ is dense inside $\mathrm{C}(Y)$ with respect to $\norm{\cdot}_{2, \Delta_2}$, there is a self-adjoint element $g \in \mathrm{C}(X)$ with $\phi(g)$ sufficiently close to $g'$ (with respect to $\norm{\cdot}_{2, \Delta_2}$) that 
$$\norm{\phi(f) - \phi(g)}_{2, \Delta_2} < \eps \quad \mathrm{and} \quad \tau(\chi_\delta(g))  < \eps,\quad \mu \in \Delta_2.$$
In particular, $\mu((g)^{-1}(0)) < \eps$ for all $\mu \in \Delta_1$. By Theorem \ref{SBP-2-norm} again, $(X, \Delta_1)$ has the (SBP).
\end{proof}

\subsection{Real rank zero}\label{rr0-section}

It is useful to characterize the (SBP) using the real rank zero property of a certain sequence algebra. Let us start with the following lemma:

\begin{lem}\label{SBP-approx}
Assume that $(X, \Delta)$ has the (SBP). Let $f\in\mathrm{C}_\Real(X)$, and let $\eps>0$. Then there are sequences $(g_n), (h_n)$ in $ \mathrm{C}_\Real(X)$ such that for each $n=1, 2, ... $, 
\begin{enumerate}
\item $\norm{f-g_n}_\infty < \eps$,
\item $\tau_\mu(\abs{g_nh_n - 1}^2) < 1/2^n$ for all $\mu\in\Delta$, 
\item $\norm{h_n}_\infty < 4/\eps$, 
\item $\norm{g_n - g_{n+1}}_{2, \Delta} < {1}/{2^n}$ and $\norm{h_n - h_{n+1}}_{2, \Delta} < {1}/{2^n}$.
\end{enumerate}
\end{lem}

\begin{proof}
Without loss of generality, one may assume that $\eps<1$ and $\norm{f} = 1$.

Choose an open cover $\{U_1, ..., U_N\}$ of $X$ such that  
$$\abs{f(x) - f(y)} < \eps/2,\quad x, y\in U_i,\ i=1, ..., N.$$

By Lemma \ref{unity-sbp}, there is a partition of unity $\phi^{(1)}_i: X \to [0, 1]$, subordinate to $U_1, U_2, ..., U_N,$ such that 
\begin{equation}\label{small-step-1}
\mu(\bigcup_{i=1}^N (\phi^{(1)}_i)^{-1}((0, 1))) < \eps/2N{ ({4}/{\eps} +5)^2 },\quad \mu\in\Delta.
\end{equation}
Consider the collection of open sets 
\begin{equation}\label{refine}
U^{(1)}_i = (\phi_i^{(1)})^{-1}((0, 1]),\quad i=1, 2, ..., N,
\end{equation} 
which is an open cover of $X$ satisfying 
\begin{equation*}
U_i^{(1)} \subseteq U_i ,\quad i=1, ..., N.
\end{equation*}
Using Lemma \ref{unity-sbp} again, one obtains 
a partition of unity $\phi^{(2)}_i: X \to [0, 1]$, subordinate to $U^{(1)}_1, U^{(1)}_2, ..., U^{(1)}_N$, such that 
\begin{equation*}
\mu(\bigcup_{i=1}^N (\phi^{(2)}_i)^{-1}((0, 1))) < \eps/2^2{ N ({4}/{\eps} +5)^2 },\quad \mu\in\Delta.
\end{equation*}
Noting that, for each $i=1, ..., N$, by \eqref{refine},  $\phi^{(2)}_i(x) = 1$ for all $x$ satisfying $\phi^{(1)}_i(x) = 1$ (so $x \notin U_j^{(1)}$, $j\neq i$), one then has 
$$(\phi_i^{(2)})^{-1}((0, 1)) \subseteq (\phi^{(1)}_i)^{-1}((0, 1)) \subseteq \bigcup_{i=1}^N (\phi^{(1)}_i)^{-1}((0, 1)) ,$$
and therefore, by \eqref{small-step-1},
$$\norm{\phi_i^{(1)} - \phi_i^{(2)} }_{2, \Delta} < \eps/2N{ ({4}/{\eps} +5)^2 } < \eps/2 N.$$

Repeating this process, one obtains partitions of unity $\phi^{(n)}_i: X \to [0, 1]$, $i=1, 2, ..., N$, subordinate to $U_1, U_2, ..., U_N$, such that 
\begin{equation}\label{small-nbhds-2}
\mu(\bigcup_{i=1}^N (\phi^{(n)}_i)^{-1}((0, 1))) < \eps/N2^n{ ({4}/{\eps} +5)^2 },\quad \mu\in\Delta,
\end{equation}
and
\begin{equation}\label{small-step-n}
\norm{\phi_i^{(n)} - \phi_i^{(n+1)} }_{2, \Delta} < {\eps}/{2^n} N .
\end{equation}

For each $i=1, ..., N$, pick $x_i\in U_i$ and then pick a real number $y_i$ such that $$\abs{y_i} > \eps/4 \quad \textrm{and}\quad \abs{y_i - f(x_i)}<\eps/2.$$
Define
$$g_n = \sum_{i=1}^N y_i\phi^{(n)}_i \quad\mathrm{and}\quad h_n = \sum_{i=1}^N \frac{1}{y_i}\phi^{(n)}_i.$$ 
Then
$$\norm{f - g_n}_\infty < \eps$$
and
$$\abs{h_n(x)} \leq \sum_{i=1}^N \frac{\phi^{(n)}_i(x)}{\abs{y_i}} < \frac{4}{\eps},\quad x\in X.$$

Note that
$(g_nh_n)(x) = 1$ whenever $x\in \bigsqcup_{i=1}^N (\phi^{(n)}_i)^{-1}(1)$, and then, together with \eqref{small-nbhds-2}, for all $\mu\in\Delta$, 
$$\tau_\mu(\abs{g_nh_n - 1}^2) < (\frac{4\norm{g_n}}{\eps} +1)^2 \cdot \frac{\eps}{2^nN(\frac{4}{\eps} +5)^2} < (\frac{4(1 + \eps)}{\eps} +1)^2 \frac{1}{2^n(\frac{4}{\eps} +5)^2}  = \frac{1}{2^n}.$$

By \eqref{small-step-n},
$$
\norm{g_n - g_{n+1}}_{2, \Delta}  =  \norm{ \sum_{i=1}^N y_i (\phi_i^{(n)} - \phi_i^{(n+1)} )}_{2, \Delta} \leq \sum_{i=1}^N \abs{y_i} \norm{ \phi_i^{(n)} - \phi_i^{(n+1)} }_{2, \Delta} 
 \leq 1/2^n
$$
and
$$
\norm{h_n - h_{n+1}}_{2, \Delta}  =  \norm{ \sum_{i=1}^N \frac{1}{y_i} (\phi_i^{(n)} - \phi_i^{(n+1)} )}_{2, \Delta} \leq \sum_{i=1}^N \abs{\frac{1}{y_i}} \norm{ \phi_i^{(n)} - \phi_i^{(n+1)} }_{2, \Delta} 
 \leq 1/2^n,
$$
as desired.
\end{proof}

Consider $l^\infty(\mathrm{C}(X))$ and consider the ideal $$J_{2, \omega, \Delta}:=\{(f_1, f_2, ...) \in l^\infty(\mathrm{C}(X)): \lim_{n\to\omega} (\sup\{\tau(\abs{f_n}^2): \tau\in\Delta\}) = 0\}.$$
In general, if $A$ is a C*-algebra and $\Delta$ is a closed set of tracial states of $A$, one also considers $l^\infty(A)$ and the ideal $$J_{2, \omega, \Delta}:=\{(a_1, a_2, ...) \in l^\infty(A): \lim_{n\to\omega} (\sup\{\tau(\abs{a_n}^2): \tau\in\Delta\}) = 0\}.$$

\begin{thm}\label{RR0=SBP}
The C*-algebra $l^\infty(\mathrm{C}(X))/J_{2, \omega, \Delta}$ has real rank zero if, and only if, $(X, \Delta)$ has the (SBP).
\end{thm}

\begin{proof}
Assume that $l^\infty(\mathrm{C}(X))/J_{2, \omega, \Delta}$ has real rank zero. Let $f\in \mathrm{C}_\Real(X)$, and let $\eps>0$ be arbitrary. Consider the image of the constant sequence $\overline{(f, f, ...)} \in l^\infty(\mathrm{C}(X))/J_{2, \omega, \Delta}.$ By the real rank zero assumption, there is an invertible self-adjoint element $\overline{(g_1, g_2, ...)} \in l^\infty(\mathrm{C}(X))/J_{2, \omega, \Delta}$ such that 
$$ \norm{\overline{(f, f, ...)} - \overline{(g_1, g_2, ...)}}_\infty <  \eps.$$ 
In particular,
\begin{equation}\label{weakly-dense}
 \norm{\overline{(f, f, ...)} - \overline{(g_1, g_2, ...)}}_{2, \omega, \Delta} <  \eps,
 \end{equation}
 where
 $$\norm{ (a_1, a_2, ...) }_{2, \omega, \Delta} = \lim_{n\to\omega} (\sup\{\tau(\abs{a_n}^2): \tau\in\Delta\}),\quad (a_1, a_2, ...) \in l^\infty(\mathrm{C}(X)). $$

Since $\overline{(g_1, g_2, ...)}$ is invertible, there is $\delta>0$ such that (with $\chi_\delta$ as in Theorem \ref{SBP-2-norm}) $$\chi_\delta(\overline{(g_1, g_2, ...)}) = 0.$$ Hence, $$\mathrm{dist}_{2, \omega, \Delta}((f-g_1, f- g_2, ... ), J_{2, \omega, \Delta}) <\eps \quad \mathrm{and} \quad (\chi_\delta(g_1), \chi_\delta(g_2), ... ) \in J_{2, \omega, \Delta},$$ and then, with $n$ sufficiently close to $\omega$, 
$$\norm{f - g_n}_{2, \Delta} < \eps \quad \mathrm{and} \quad \tau(\chi_\delta(g_n))< \eps,\quad \tau \in\Delta.$$
By Theorem \ref{SBP-2-norm}, $(X, \Delta)$ has the (SBP).

Now, assume that $(X, \Delta)$ has the (SBP), and let us show that $l^\infty(\mathrm{C}(X))/J_{2, \omega, \Delta}$ has real rank zero. Let $f\in l^\infty(\mathrm{C}(X))/J_{2, \omega, \Delta}$ be a self-adjoint element, and let $\eps>0$ be arbitrary. Pick a self-adjoint representative $(f_1, f_2, ...)$ of $f$. By Lemma \ref{SBP-approx} (with $f=f_n$, $n=1, 2, ...$), there are self-adjoint elements $g_1, g_2, ...$ and $h_1, h_2, ...$ in $\mathrm{C}(X)$ such that for each $n=1, 2, ...$, 
\begin{enumerate}
\item $\norm{f_n-g_n}_\infty < \eps$,
\item $\tau_\mu(\abs{g_nh_n - 1}^2) < 1/n$ for all $\mu\in\Delta$, and
\item $\norm{h_n}_\infty < 4/\eps$.
\end{enumerate}
Then, with $g:=\overline{(g_1, g_2, ...)}$ and $h:=\overline{(h_1, h_2, ...)}$ in $l^\infty(\mathrm{C}(X))/J_{2, \omega, \Delta}$, one has
$$\norm{f - g}_\infty < \eps\quad\mathrm{and}\quad gh = 1.$$ That is, $g$ is invertible. Since $\eps$ is arbitrary, this shows that $l^\infty(\mathrm{C}(X))/J_{2, \omega, \Delta}$ has real rank zero.
\end{proof}

Note that in the proof of $\mathrm{(RR0)} \Rightarrow \mathrm{(SBP)}$ above, it is the following weak real rank zero property that is used: each self-adjoint element is approximated in the trace norm $\norm{\cdot}_{2, \omega, \Delta}$ by invertible self-adjoint elements (see \eqref{weakly-dense}). Therefore, we have
\begin{cor}
The pair $(X, \Delta)$ has the (SBP) if (and only if) the invertible self-adjoint elements of $l^\infty(\mathrm{C}(X))/J_{2, \omega, \Delta}$ are dense in the set of self-adjoint elements with respect to $\norm{\cdot}_{2, \omega, \Delta}$.
\end{cor}

Thus, we also have the following (possibly slightly surprising) corollary:
\begin{cor}
The C*-algebra $l^\infty(\mathrm{C}(X))/J_{2, \omega, \Delta}$ has real rank zero if (and only if) its invertible self-adjoint elements are dense in the set of self-adjoint elements with respect to $\norm{\cdot}_{2, \omega, \Delta}$.
\end{cor}

The following corollary (directly of Theorem \ref{RR0=SBP}) is straightforward:
\begin{cor}
Consider $(X, \Delta_1)$ and $(Y, \Delta_2)$ such that $l^\infty(\mathrm{C}(X))/J_{2, \omega, \Delta_1} \cong l^\infty(\mathrm{C}(Y))/J_{2, \omega, \Delta_2}$. Then $(X, \Delta_1)$ has the (SBP) if, and only if, $(Y, \Delta_2)$ has the (SBP).
\end{cor}

Inside $l^\infty(\mathrm{C}(X))/J_{2, \omega, \Delta}$, consider the sub-C*-algebra consisting of the classes of bounded sequences of $\mathrm{C}(X)$ which are $\norm{\cdot}_{2, \Delta}$-Cauchy. This C*-algebra is the closure of $\mathrm{C}(X)$ inside $l^\infty(\mathrm{C}(X))/J_{2, \omega, \Delta}$  under the norm $\norm{\cdot}_{2, \omega,\Delta}$. Denote it by $\overline{\mathrm{C}(X)}^{\Delta}$ and denote the restriction of $\norm{\cdot}_{2, \omega,\Delta}$ to $\overline{\mathrm{C}(X)}^{\Delta}$ still by $\norm{\cdot}_{2, \Delta}$.  


\begin{thm}\label{RR0=SBP-X}
If the pair $(X, \Delta)$ has the (SBP), then any constant sequence in $\mathrm{C}_{\Real}(X)$ can be approximated in $\norm{\cdot}_\infty$ by self-adjoint invertible elements of $\overline{\mathrm{C}(X)}^{\Delta}$.

On the other hand, if any constant sequence in $\mathrm{C}_{\Real}(X)$ can be approximated in $\norm{\cdot}_{2, \Delta}$ by self-adjoint invertible elements of $\overline{\mathrm{C}(X)}^{\Delta}$, then $(X, \Delta)$ has the (SBP).
\end{thm}

\begin{proof}
Assume that any self-adjoint element of $\mathrm{C}(X)\subseteq \overline{\mathrm{C}(X)}^{\Delta}$ is in the $\norm{\cdot}_{2, \omega, \Delta}$ closure of invertible self-adjoint elements. Let $f\in \mathrm{C}_\Real(X)$, and let $\eps>0$ be arbitrary. Consider the image of the constant sequence $\overline{(f, f, ...)} \in l^\infty(\mathrm{C}(X))/J_{2, \omega, \Delta}.$ By the assumption, there is an invertible self-adjoint element $\overline{(g_1, g_2, ...)} \in l^\infty(\mathrm{C}(X))/J_{2, \omega, \Delta}$ such that 
$$ \norm{\overline{(f, f, ...)} - \overline{(g_1, g_2, ...)}}_{2, \omega, \Delta} <  \eps.$$ Then there is $\delta>0$ such that $$\chi_\delta(\overline{(g_1, g_2, ...)}) = 0.$$ 
Hence $$\mathrm{dist}_{2, \omega, \Delta}((f-g_1, f- g_2, ... ), J_{2, \omega, \Delta}) <\eps \quad \mathrm{and} \quad (\chi_\delta(g_1), \chi_\delta(g_2), ... ) \in J_{2, \omega, \Delta},$$
and so, with sufficiently large $n$, one has
$$\norm{f - g_n}_{2, \Delta} < \eps \quad \mathrm{and} \quad \tau(\chi_\delta(g_n))< \eps,\quad \tau \in\Delta.$$
By Theorem \ref{SBP-2-norm}, $(X, \Delta)$ has the (SBP).

Now, assume that $(X, \Delta)$ has the (SBP). It follows from Lemma \ref{SBP-approx} that a constant sequence $(f)$ with $f \in \mathrm{C}_\Real(X)$ can be approximated by an invertible element of $\overline{\mathrm{C}(X)}^{\Delta}$ to within $\eps$,  as the sequences $(g_n)$ and $(f_n)$ are $\norm{\cdot}_{2, \Delta}$-Cauchy. 
\end{proof}

Since $\norm{\cdot}_{2, \Delta}$ is dominated by $\norm{\cdot}_{\infty}$, we have the following characterization of the (SBP):
\begin{cor}
The pair $(X, \Delta)$ has the (SBP) if, and only if, the invertible self-adjoint elements of $\overline{\mathrm{C}(X)}^{\Delta}$ are dense in the set of self-adjoint elements with respect to $\norm{\cdot}_{2, \omega, \Delta}$.
\end{cor}

\begin{rem}
In contrast to Theorem \ref{RR0=SBP}, we do not know whether the property (SBP) of $(X, \Delta)$ is characterized by the property (RR0) of $\overline{\mathrm{C}(X)}^{\Delta}$.
\end{rem}

\section{A stronger version of uniform property $\Gamma$ and the small boundary property}

In this section, let us use the characterization of the (SBP) of Theorem \ref{SBP-2-norm} to connect it to a notion of approximate divisibility on traces (Theorem \ref{Gamma2SBP}).

Recall (\cite{CETW-Gamma}) that a C*-algebra $A$ is said to have the uniform McDuff property if for each $n$, there is a unital embedding
$$\mathrm{M}_n(\Comp) \to (l^\infty(A)/J_{2, \omega, \Delta})\cap A',$$
and the C*-algebra $A$ is said to have uniform property $\Gamma$ if for each $n$, there is a partition of unity $$p_1, p_2, ..., p_n \in (l^\infty(A)/J_{2, \omega, \Delta})\cap A'$$ such that
$$\tau(p_iap_i) = \frac{1}{n}\tau(a),\quad a\in A,\ i=1, 2, ..., n,\  \tau \in \tr(A)_\omega,$$
where $\tr(A)_\omega$ is the set of traces of $l^\infty(A)$ of the form
$$\tau((a_i)) = \lim_{i\to\omega}\tau_i(a_i),\quad \tau_i \in \tr(A),$$
and $a$ is regarded as the constant sequence $(a) \in l^\infty(A)$.

Let us introduce the following stronger versions of these properties, restricting the partitions to belong to a subalgebra $D \subseteq A$: 

\begin{defn}\label{RGamma-defn}
Let $A$ be a C*-algebra and let $D\subseteq A$ be a sub-C*-algebra. 
The pair $(D, A)$ is said to have the strong uniform McDuff property if for each $n\in \mathbb N$, there is a unital embedding
$$\phi: \mathrm{M}_n(\Comp) \to (l^\infty(A)/J_{2, \omega, \Delta})\cap A'$$
such that
$$ \phi(e_{ii}) \in  l^\infty(D)/J_{2, \omega, \Delta},\quad i=1, ..., n. $$

The pair $(D, A)$ is said to have the strong uniform property $\Gamma$ if for each $n\in \mathbb N$, there is a partition of unity $$p_1, p_2, ..., p_n \in (l^\infty(D)/J_{2, \omega, \Delta}) \cap A'$$ such that
$$\tau(p_iap_i) = \frac{1}{n}\tau(a),\quad a\in A,\ \tau\in \tr(A)_\omega,$$
where $\tr(A)_\omega$ is as above.
\end{defn}

Let us now formulate weaker versions of these two properties, relaxing the approximate centrality of $p_1, ..., p_n$, or removing the algebra $A$ altogether.


\begin{defn}\label{Definition-WGamma}
A pair $(D, A)$, where $A$ is a unital C*-algebra and $D\subseteq A$ is a unital subalgebra, will be said to be (tracially) approximately divisible if there is $K>0$ such that for each $n\in \mathbb N$, there is a partition of unity $$p_1, p_2, ..., p_n \in l^\infty(D)/J_{2, \omega, \Delta}$$ such that
$$\tau(p_iap_i) \leq \frac{1}{n}K\tau(a),\quad a\in D^+,\ \tau\in \tr(A)_\omega,\ i=1, ..., n.$$

Similarly, a pair $(D, \Delta)$ (or the pair $(X, \Delta)$), where $D = \mathrm{C}(X)$ and $\Delta$ is a closed set of Borel probability measures, is said to be approximately divisible if there is $K>0$ such that for each  $n \in \mathbb N$, there is a partition of unity $$p_1, p_2, ..., p_n \in l^\infty(D)/J_{2, \omega, \Delta}$$ such that
$$\tau(p_iap_i) \leq \frac{1}{n} K \tau(a),\quad a\in D^+,\ \tau\in \Delta_\omega,\ i=1, ..., n,$$
where $\Delta_\omega$ is the set of traces of $l^\infty(D)$ of the form
$$\tau((a_i)) = \lim_{i\to\omega}\tau_i(a_i),\quad \tau_i \in \Delta.$$

\end{defn}

\begin{rem}
It is clear that the strong uniform McDuff property of $(D, A)$ implies the strong uniform property $\Gamma$ (and in particular approximate divisibility).
\end{rem}

The property of approximate divisibility can be formulated locally as follows:
\begin{lem}\label{loc-Gamma}
The pair $(D, \Delta)$ is approximately divisible if, and only if, there is $K>0$ such that for any finite set $\mathcal F \subseteq D^+$, any $\eps>0$, and any $n\in\mathbb N$, there are positive contractions $p_1, ..., p_n \in D$ such that
\begin{enumerate}
\item $\norm{p_ip_j}_{2, \Delta} <\eps $ if $i\neq j$,
\item $\norm{p_i -p_i^2}_{2, \Delta} < \eps$, $i=1, ..., n$,
\item $\norm{1-(p_1^2 + \cdots + p_n^2)}_{2, \Delta} < \eps$,
\item $\tau(p_ifp_i) <  \frac{1}{n}K\tau(f) + \eps$,  $i=1, ..., n$, $f\in\mathcal F$, $\tau \in \Delta$.
\end{enumerate}
\end{lem}

It turns out that approximate divisibility implies the (SBP):
\begin{thm}\label{Gamma2SBP}
Let $D$ be a commutative C*-algebra and let $\Delta$ be a closed set of probability Borel measures. If $(D, \Delta)$ is approximately divisible, then $(D, \Delta)$ has the (SBP).  

The converse holds in the case that $\Delta$ is $\tr(D \rtimes \Gamma)$ restricted to $D \subseteq D \rtimes \Gamma$, where $\Gamma$ is a discrete amenable group acting freely on $D$. 
\end{thm}

Let us start with some preparations:

\begin{lem}
Let $f \in \mathrm{C}(X)$ be a self-adjoint element, and let $\eps>0$. Then there exist $n\in \mathbb N$, and self-adjoint elements $f_1, f_2, ..., f_n \in \mathrm{C}(X)$, such that
\begin{enumerate}
\item $\norm{f - f_i}_\infty < \eps$, $i=1, 2, ..., n$, and
\item $$\frac{1}{n} |\{ 1 \leq i\leq n:  f_i(x) =0  \} | < \eps,\quad x \in X.$$
\end{enumerate} 
\end{lem}

\begin{proof}
Choose an open cover $\mathcal U$ of $X$ such that
$$\abs{f(x) - f(y)} < \eps,\quad x, y \in U,\ U \in\mathcal U.$$
Choose a partition of unity $\{\phi_U: U\in \mathcal U\}$, subordinate to $\mathcal U$, and choose $x_U \in U$ for each $U \in \mathcal U$. Consider the function
$$g:=\sum_{U \in \mathcal U} f(x_U)\phi_U,$$
and it is straightforward to verify that $$\norm{ f - g }_\infty < \eps.$$ 
It is straightforward to check that $g$ has the factorization
\begin{displaymath}
\xymatrix{
X \ar[r]^-{\Phi} & N(\mathcal U) \ar[r]^-{\Psi} & \Real}
\end{displaymath}
where $N(\mathcal U)$ is the nerve complex of $\mathcal U$ (see, for instance, Definition IX 2.2 of \cite{ES-book} for the definition of the nerve complex of an open cover), the map $\Phi: X \to N(\mathcal U)$ is given by 
$$ x \mapsto \sum_{U \in \mathcal U} \phi_U(x)[U],$$
and the map $\Psi: N(\mathcal U) \to \Real$ is the linear map
$$\sum_{U \in \mathcal U}\alpha_U [U] \mapsto \sum_{U \in \mathcal U} f(x_U)\alpha_U.$$

Pick $n > \mathrm{dim}(N(\mathcal U))/\eps$, and consider the map 
$$(\Psi, ..., \Psi): N(\mathcal U) \to \Real^n.$$ By Lemma 5.7 of \cite{GLT-Zk}, there is a map  $$(\Psi_1, ...  \Psi_n): N(\mathcal U) \to \Real^n$$ such that
$$\norm{\Psi - \Psi_i}_\infty < \eps,\quad 1 \leq i\leq n$$
and
$$|\{ 1 \leq i \leq n: \Psi_i(y) = 0  \}| < \mathrm{dim}(N(\mathcal U)),\quad y \in N(\mathcal U).$$
Then $$(f_1, ..., f_n) := (\Psi_1\circ \Phi, ..., \Psi_n\circ\Phi)$$
satisfies  the conclusion of the lemma.
\end{proof}

\begin{rem}
The open cover $\mathcal U$ can be chosen to have order $1$, and hence the simplex $N(\mathcal U)$ to have dimension $1$. Therefore, the number $n$ in the proof can be chosen to be $\lfloor 1/\eps \rfloor + 1$, which is independent of $X$.
\end{rem}

\begin{cor}\label{mult-pert}
Let $f \in \mathrm{C}(X)$ be a self-adjoint element, and let $\eps>0$. Then there exist $n\in \mathbb N$, $\delta>0$, and self-adjoint elements $f_1, f_2, ..., f_n \in \mathrm{C}(X)$ such that
\begin{enumerate}
\item $\norm{f - f_i}_\infty < \eps$, $i=1, 2, ..., n$, and
\item $$\frac{1}{n} |\{ 1 \leq i\leq n:  |f_i(x) | < \delta  \} | < \eps,\quad x\in X.$$
\end{enumerate} 
In particular,
$$\frac{1}{n}(\tau((f_1)_\delta) + \cdots + \tau((f_n)_\delta)) < \eps,\quad \tau\in\mathrm{T}(\mathrm{C}(X)),$$
where $(f)_\delta = \chi_\delta(f)$ with $\chi_\delta: \Real \to\Real$ defined by
$$\chi_\delta(t) = 
\left\{
\begin{array}{ll}
1, & |t| < \delta/2, \\
0, & |t| > \delta, \\
\textrm{linear}, & \textrm{otherwise}.
\end{array}
\right.
$$
\end{cor}

\begin{proof}
By the lemma above, there exist $n\in \mathbb N$, and self-adjoint elements $f_1, f_2, ..., f_n \in \mathrm{C}(X)$ such that
\begin{enumerate}
\item $\norm{f - f_i}_\infty < \eps$, $i=1, 2, ..., n$, and
\item $$\frac{1}{n} |\{ 1 \leq i\leq n:  |f_i(x) | =0  \} | < \eps,\quad x\in X.$$
\end{enumerate} 
Since $f_1, ...  f_n$ are continuous, for each $x\in X$, there exist a neighbourhood $U_x \ni x$ and $\delta_x>0$ such that
$$\frac{1}{n} |\{ 1 \leq i\leq n:  |f_i(y) | < \delta_x  \} | < \eps,\quad y\in U_x.$$
Then the corollary follows from the compactness of $X$.
\end{proof}

It is worth pointing out that the corollary above holds for a general unital C*-algebra.

\begin{cor}\label{mult-pert-general}
Let $A$ be a unital C*-algebra. Let $f \in A$ be a self-adjoint element, and let $\eps>0$. Then there exist $n\in \mathbb N$, $\delta>0$, and self-adjoint elements $f_1, f_2, ..., f_n \in A$ such that
\begin{enumerate}
\item $\norm{f - f_i}_\infty < \eps$, $i=1, 2, ..., n$, and
\item $$\frac{1}{n}(\tau((f_1)_\delta) + \cdots + \tau((f_n)_\delta)) < \eps,\quad \tau\in\mathrm{T}(A). $$
\end{enumerate} 
\end{cor}
\begin{proof}
The statement follows directly from Corollary \ref{mult-pert} applied to the sub-C*-algebra $D:=\textrm{C*}\{1, f\}$.
\end{proof}

\begin{proof}[Proof of Theorem \ref{Gamma2SBP}]
To show that the pair $(D, \Delta)$ has the (SBP), by Corollary \ref{SBP-2-norm}, it is enough to show that for any self-adjoint element $f \in D$ and any $\eps>0$, there is a self-adjoint element $g \in D$ such that
\begin{enumerate}
\item $\| f - g\|_{2, \Delta} < \eps$, and
\item $\mu(g^{-1}(0)) < \eps$, $\mu \in \Delta$.
\end{enumerate}

By Corollary \ref{mult-pert}, for the given $\eps$, there exist $n\in\mathbb N$ and self-adjoint elements $$f_1, f_2, ..., f_n \in D$$  such that
\begin{equation}\label{mult-pert-eq-1}
\norm{f - f_i}_\infty < \eps,\quad i=1, 2, ..., n, 
\end{equation}
and 
there is $\delta>0$ such that 
\begin{equation*}
\frac{1}{n}(\tau((f_1)_\delta) + \cdots + \tau((f_n)_\delta)) < \eps/K,\quad \tau\in\mathrm{T}(\mathrm{C}(X)),
\end{equation*} 
where $K$ is the constant of approximate divisibility. In particular, regarding $(f_1)_\delta, ..., (f_n)_\delta$ as constant sequences in $D$, we have
\begin{equation}\label{mult-pert-eq-2}
\frac{1}{n}(\tau((f_1)_\delta) + \cdots + \tau((f_n)_\delta)) < \eps/K,\quad \tau\in\Delta_\omega.
\end{equation} 

By approximate divisibility, there are $$p_1, p_2, ..., p_n \in l^\infty(D)/J_{2, \omega, \Delta} $$
such that
\begin{equation}\label{r-gamma-cut-eq}
\tau(p_iap_i) \leq \frac{1}{n} K \tau(a),\quad a\in D^+,\ \tau\in \Delta_\omega.
\end{equation}

Consider the element $$g:=p_1\overline{(f_1)}p_1 + \cdots + p_n\overline{(f_n)}p_n \in l^\infty(D)/J_{2, \omega, \Delta}. $$
By \eqref{mult-pert-eq-1},
\begin{equation}\label{close-cond-1}
\norm{f-g}_{2, \omega, \Delta} = \norm{p_1\overline{(f-f_1)}p_1 + \cdots + p_n\overline{(f-f_n)}p_n}_{2, \omega, \Delta}  < \eps.
\end{equation}
Note that, for each $\tau \in \Delta_\omega$, by \eqref{r-gamma-cut-eq},
$$\tau(p_i\overline{((f_i)_\delta)} p_i) \leq \frac{1}{n}K\tau((f_i)_\delta),\quad i=1, ..., n, \ \tau\in\Delta_\omega,$$
and hence, together with \eqref{mult-pert-eq-2},
\begin{eqnarray}\label{close-cond-2}
\tau((g)_\delta) & = & \tau(p_1\overline{((f_1)_\delta)}p_1) + \cdots + \tau(p_n\overline{((f_n)_\delta)}p_n) \\
& \leq & \frac{1}{n}K(\tau((f_1)_\delta) + \cdots + \tau((f_n)_\delta)  ) \nonumber \\
& < & \eps.  \nonumber
\end{eqnarray}
Pick a representative sequence $g = \overline{(g_k)}$ with $g_k$, $k=1, 2, ...$, self-adjoint elements of $D$. By \eqref{close-cond-1} and \eqref{close-cond-2}, with some sufficiently large $k$, the function $g_k$ satisfies
\begin{enumerate}
\item $\norm{f - g_k}_{2, \Delta} < \eps$, and
\item $\tau((g_k)_\delta) < \eps$, $\tau\in \Delta$,
\end{enumerate}
as desired.

Now, assume that $(D, \tr(A))$ has the SBP (Definition \ref{defn-sbp}), where $A= D \rtimes\Gamma$ and $(D, \Gamma)$ is free and $\Gamma$ is amenable. It follows from the proof of Theorem 9.4 of \cite{KS-comparison} that $(D, A)$ has the strong uniform property $\Gamma$, and so $(D, \tr(A))$ is approximately divisible.
\end{proof}

\begin{cor}
Consider a C*-pair $(D, A)$. If $A$ has the strong uniform McDuff property, then $(D, \tr(A))$ has the (SBP).
\end{cor}

\begin{cor}
Consider C*-pairs $(D_1, A_1)$ and $(D_2, A_2)$, and assume that one of them, say $(D_1, \tr(A_1))$, is approximately divisible, then $(D_1 \otimes D_2, \mathrm{T}(A_1 \otimes A_2))$ has the (SBP).
\end{cor}

\begin{proof}
Note that $\partial \tr(A_1 \otimes A_2) = \partial\tr(A_1) \times \partial\tr(A_2)$. Then, using Lemma \ref{loc-Gamma}, it is easy to see that $D_1\otimes D_2$ is approximately divisible with respect to $\tr(A_1\otimes A_2)$ (say, if $(D_1, A_1)$ is approximately divisible, then consider the corresponding elements $p_1\otimes 1_{D_2}$, ..., $p_n\otimes 1_{D_2}$ in $D_1\otimes D_2$).
\end{proof}

\begin{rem}
Is $(D_1 \otimes D_2, A_1 \otimes A_2)$ always approximately divisible. (Cf. Theorem 5.6 of \cite{EN-MD0}) 
\end{rem}

\section{The case that $\mathrm{T}(A)$ is a Bauer simplex}


As an application of Theorem \ref{Gamma2SBP}, in this section, let us show that if $A = \mathrm{C}(X)\rtimes \Int^d$ or $A$ is an AH algebra with diagonal maps, and if $\partial \mathrm{T}(A)$ is compact, then $A$ has the (SBP) (Theorem \ref{Bauer-SBP}), and hence $A$ is $\mathcal Z$-stable.

%


\begin{defn}
Let $D$ be a unital commutative C*-algebra, and let $\Delta \subseteq \mathrm{T}(D) $ be a closed subset. Let $c = \overline{(c_1, c_2, ... )} \in l^\infty(D)/J_{2, \omega, \Delta}$ be a non-zero positive contraction. Then $\Delta$ is said to be ample with respect to $c$ if for any $\tau_1, \tau_2, ... \in \Delta$, the limit tracial states
\begin{equation}\label{tr-1}
 D \ni x \mapsto \frac{1}{(\tau_n)_\omega(c)}(\tau_n)_\omega(x c) = \frac{\lim_{n\to\omega}\tau_n(xc_n)}{\lim_{n\to\omega} \tau_n(c_n)} \in \Comp
 \end{equation}
and
\begin{equation}\label{tr-2}
D \ni x \mapsto \frac{1}{1 - (\tau_n)_\omega(c)}(\tau_n)_\omega(x (1-c)) = \frac{\lim_{n\to\omega}\tau_n(x(1-c_n))}{1- \lim_{n\to\omega} \tau_n(c_n)} \in \Comp
\end{equation}
are still in $\Delta$ (when $(\tau_n)_\omega(c) = 0$ or $1$, only the trace \eqref{tr-2} or \eqref{tr-1}, respectively, should be considered), where $(\tau_n)_\omega$ is the limit trace of $(\tau_n)$ on $l^\infty(D)$.
\end{defn}

\begin{example}
Let $A$ be a unital C*-algebra, and let $D \subseteq A$ be a unital commutative sub-C*-algebra. 
Then the set $\tr(A)|_D$ is ample with respect to any positive contraction $$c \in (l^\infty(D)/J_{2, \omega, \tr(A)|_D}) \cap A',$$ where $A$ is regarded as the subalgebra of $l^\infty(A)/J_{2, \omega, \tr(A)}$ consisting of constant sequences,  as then, the asymptotic commutativity implies that 
\begin{equation*}
 A \ni x \mapsto \frac{1}{(\tau_n)_\omega(c)}(\tau_n)_\omega(x c) = \frac{\lim_{n\to\omega}\tau_n(xc_n)}{\lim_{n\to\omega} \tau_n(c_n)} \in \Comp
 \end{equation*}
 and
 \begin{equation*}
A \ni x \mapsto \frac{1}{1 - (\tau_n)_\omega(c)}(\tau_n)_\omega(x (1-c)) = \frac{\lim_{n\to\omega}\tau_n(x(1-c_n))}{1- \lim_{n\to\omega} \tau_n(c_n)} \in \Comp
\end{equation*}
 are tracial states of $A$.
 
\end{example}

\begin{lem}\label{extreme-cond}
Consider a pair $(D, \Delta)$ with $\Delta\subseteq \mathrm{T}(D)$ a closed convex subset with $\partial \Delta$ compact.
If $\Delta$ is ample with respect to a positive contraction $c = \overline{(c_n)} \in l^\infty(D)/J_{2, \omega, \Delta}$, 
then $$\lim_{k\to\omega} \sup_{\tau\in \partial\Delta}\abs{\tau(ac_k) - \tau(a)\tau(c_k)} = 0,\quad a\in D.$$
\end{lem}

\begin{proof}
If the statement were not true, then there would exist $a\in D$, $\eps>0$ and a sequence $(\tau_n)$ in $\Delta$ such that
$$\abs{\tau_n(ac_n) - \tau_n(a)\tau_n(c_n))} \geq \eps,\quad n = 1, 2, ... .$$ Consider $(\tau_n)_\omega \in \Delta_\omega$. Then \begin{equation}\label{controdition-eq-1}
(\tau_n)_\omega(ac) \neq (\tau_n)_\omega(a)(\tau_n)_\omega(c),
\end{equation}
and, in particular, 
$$(\tau_n)_\omega(c) \neq 0, 1.$$

Note that, for any $x \in D$,
\begin{equation}\label{lin-comb}
(\tau_n)_\omega(x) = (\tau_n)_\omega(c) \cdot \frac{(\tau_n)_\omega(xc)}{(\tau_n)_\omega(c)}  +  (1 - (\tau_n)_\omega(c)) \cdot \frac{(\tau_n)_\omega(x(1-c))}{1-(\tau_n)_\omega(c)}.
\end{equation}

Since $\Delta$ is ample for $c$, we have 
$$\frac{(\tau_n)_\omega(\ \cdot\ c)}{(\tau_n)_\omega(c)},\quad  \frac{(\tau_n)_\omega(\ \cdot \ (1-c))}{1-(\tau_n)_\omega(c)} \in \Delta.$$
Since $(\tau_n) \subseteq \partial \Delta$ and $\partial\Delta$ is compact, $(\tau_n)_\omega|_D \in \partial \Delta$; in particular, $(\tau_n)_\omega|_D$ is an  extreme point of $\Delta$. Therefore (note that $(\tau_n)_\omega(c) \neq 0$), by \eqref{lin-comb}, 
$$(\tau_n)_\omega(x) = \frac{(\tau_n)_\omega(x c)}{(\tau_n)_\omega(c)}, \quad x \in D,$$
which is
$$(\tau_n)_\omega(xc) = (\tau_n)_\omega(x) (\tau_n)_\omega(c), \quad x\in D.$$
This contradicts \eqref{controdition-eq-1}. 
\end{proof}


\begin{thm}\label{Bauer-case-A-free}
Consider a pair $(D, \Delta)$ with $\Delta\subseteq \tr(D)$ a closed convex subset with $\partial \Delta$ compact. Assume that, for each $n \in \mathbb N$, there is a partition of unity $p_1, p_2, ..., p_n \in l^\infty(D)/J_{2, \omega, \Delta}$ such that 
\begin{equation}\label{u-cut-1-A-free}
\tau(p_i) = \frac{1}{n},\quad \tau\in\Delta_\omega,\ i=1, ..., n,
\end{equation}
where (as before) $\Delta_\omega$ is the set of traces of $l^\infty(D)$ of the form
$$\tau((a_i)) = \lim_{i\to\omega}\tau_i(a_i),\quad \tau_i \in \Delta,$$ 
and $\Delta$ is ample with respect to each $p_i$, $i=1, ..., n$.
Then, the partition of unity $p_1, p_2, ..., p_n$ above satisfies
\begin{equation}\label{u-gamma-A-free}
\tau(p_iap_i) = \frac{1}{n}\tau(a),\quad a\in D,\quad \tau\in\Delta_\omega,\ i=1, ..., n.
\end{equation}
So, $(D, \Delta)$ is approximately divisible.
\end{thm}

\begin{proof}
Let $a \in D$ and $1 \leq i\leq n$ be arbitrary. Pick a representative $(c_k) \in l^\infty(D)$ for $p_i$. By Lemma \ref{extreme-cond},
$$\lim_{k\to\omega} \sup_{\tau\in \partial\Delta}\abs{\tau(ac_k) - \tau(a)\tau(c_k)} = 0,\quad a\in D.$$ Together with \eqref{u-cut-1-A-free}, one has
$$\lim_{k\to\omega} \sup_{\tau\in \partial\Delta}\abs{\tau(ac_k) - \tau(a)\frac{1}{n}} = 0,\quad a\in D,$$
which implies  \eqref{u-gamma-A-free}.
\end{proof}

The theorem above in particular applies to a C*-pair $(D, A)$ where $\Delta = \tr(A)|_D$. In this case, one actually only has to assume $\tr(A)$ to be Bauer (in general this does not obviously imply $\tr(A)|_D$ is Bauer), and the proof is even shorter:

\begin{thm}\label{Bauer-case}
Consider a C*-pair $(D, A)$ where $D$ is a unital commutative sub-C*-algebra of $A$, and assume that $\mathrm{T}(A)$ is a Bauer simplex. If $p_1, p_2, ..., p_n \in (l^\infty(D)/J_{2, \omega, \mathrm{T}(A)|_D}) \cap A'$, where $n\in \mathbb N$, is a partition of unity such that 
\begin{equation*}
\tau(p_i) = \frac{1}{n},\quad \tau\in \mathrm{T}(A)_\omega,\ i=1, ..., n,
\end{equation*}
where $\mathrm{T}(A)_\omega$ is the set of traces of $l^\infty(A)$ of the form
$$\tau((a_i)) = \lim_{i\to\omega}\tau_i(a_i),\quad \tau_i \in \mathrm{T}(A),$$
then,
\begin{equation}\label{u-gamma}
\tau(p_iap_i) = \frac{1}{n}\tau(a),\quad a\in A,\quad \tau\in \mathrm{T}(A)_\omega,\ i=1, ..., n.
\end{equation}
\end{thm}

\begin{proof}

Let $a \in A$ and $1 \leq i\leq n$ be arbitrary. Let $p_1, p_2, ..., p_n \in (l^\infty(D)/J_{2, \omega, \mathrm{T}(A)|_D}) \cap A'$ be a partition of unity such that 
\begin{equation}\label{u-cut-1}
\tau(p_i) = \frac{1}{n},\quad \tau\in\mathrm{T}(A)_\omega,\ i=1, ..., n.
\end{equation}

Pick a representative $(b_k) \in l^\infty(D)$ for $p_i$. Since $p_i \in (l^\infty(D)/J_{2, \omega, \mathrm{T}(A)|_D}) \cap A'$, one has
$$\lim_{k\to\omega}\norm{ab_k - b_ka}_{2, \tr(A)} = 0,\quad a\in A.$$ Since $\mathrm{T}(A)$ is a Bauer simplex,  by Proposition 3.1 of \cite{CETW-Gamma}, 
$$\lim_{k\to\omega} \sup_{\tau\in \partial\tr(A)}\abs{\tau(ab_k) - \tau(a)\tau(b_k)} = 0,\quad a\in A.$$
By \eqref{u-cut-1},
$$\lim_{k\to \omega}\sup_{\tau\in\tr(A)}\abs{\tau(b_k) - \frac{1}{n}} = 0,$$
and hence
$$\lim_{k\to\omega} \sup_{\tau\in \partial\tr(A)}\abs{\tau(ab_k) - \frac{1}{n}\tau(a)} = 0,\quad a\in A,$$
which implies \eqref{u-gamma}.
\end{proof}


The uniform Rokhlin property (URP) and the property of comparison of open sets (COS) were introduced in \cite{Niu-MD-Z}. Every free and minimal dynamical system $(X, \Int^d)$ has the (URP) and (COS) (\cite{Niu-MD-Zd}); in the case that $\Gamma$ is finitely generated and has subexponential growth, if a free and minimal dynamical system $(X, \Gamma)$ has a Cantor factor, then it has the (URP) and (COS) (\cite{Niu-MD-Z}). For free and minimal dynamical systems $(X, \Gamma)$ with the (URP) and (COS), if $(X, \Gamma)$ has the (SBP), then the C*-algebra $\mathrm{C}(X)\rtimes\Gamma$ is $\mathcal Z$-absorbing (\cite{Niu-MD-Z-absorbing}).

\begin{thm}\label{Bauer-SBP}
Let $(X, \Gamma)$ be a topological dynamical system with the (URP), where $\Gamma$ is an infinite  discrete amenable group. Assume that $\mathcal M_1(X, \Gamma)$ is a Bauer simplex (but with no restriction on $\mathrm{dim}(\partial \mathcal M_1(X, \Gamma))$). Then $(X, \Gamma)$ has the (SBP). 

In particular, if $(X, \Gamma)$ is free and minimal, and has the (URP) and (COS), the C*-algebra $A:=\mathrm{C}(X)\rtimes\Gamma$ is $\mathcal Z$-absorbing, i.e., $A \cong A \otimes \mathcal Z$.
\end{thm}

\begin{proof}

Since $(X, \Gamma)$ has the property (URP), $\mathcal M_1(X, \Gamma)$ is canonically isomorphic to $\tr(A)$. Therefore $\tr(A)$ is a Bauer simplex.

Since $(X, \Gamma)$ has the (URP) and since any discrete amenable group has tilings by F{\o}lner sets (Theorem 4.3 of \cite{DHZ-tiling}), an argument using Lemma 7.1 of \cite{LN-sr1} shows that, for any $n \in \mathbb N$, there is a partition of unity $p_1, p_2, ..., p_n \in (l^\infty(D)/J_{2, \omega, D}) \cap A'$ such that 
\begin{equation*}
\tau(p_i) = \frac{1}{n},\quad \tau\in (\mathrm{T}(A))_\omega,\ i=1, ..., n.
\end{equation*}
Then, by Theorem \ref{Bauer-case-A-free} with $\Delta = \tr(A)|_D$ (or just Theorem \ref{Bauer-case}), one has that $(D, \tr(A)|_D)$ is approximately divisible, and by Theorem \ref{Gamma2SBP}, $(D, \tr(A)|_D)$ has the (SBP).

The second statement then follows from Theorem 5.4 of \cite{Lindenstrauss-Weiss-MD} and Theorem 4.8 of \cite{Niu-MD-Z-absorbing}.
\end{proof}

\begin{rem}
If $\abs{\Gamma} < \infty$ and $(X, \Gamma)$ is minimal, then $X$ must be a finite set, and hence $(X, \Gamma)$ has the (SBP). But the statement does not hold in general without minimality for a finite group $\Gamma$, as the action of the trivial group on $[0, 1]$ provides a counterexample.
\end{rem}


Since any free and minimal dynamical system $(X, \Int^d)$ has the (URP) and (COS) (Theorem 4.2 and Theorem 5.5 of \cite{Niu-MD-Zd}), one has the following corollary:
\begin{cor}
Let $(X, \Int^d)$ be a free and minimal topological dynamical system. If $\mathcal M_1(X, \Int^d)$ is a Bauer simplex, then $(X, \Int^d)$ has the (SBP) and the C*-algebra $\mathrm{C}(X) \rtimes \Int^d$ is $\mathcal Z$-absorbing.

\end{cor}

Now, let us consider $A$, an AH algebra with diagonal maps, which is an inductive limit
$$\xymatrix{
A_1 \ar[r] & A_2 \ar[r] & \cdots \ar[r] & A = \varinjlim A_n,
}$$
where $A_i = \bigoplus_{j} \mathrm{M}_{n_{i, j}}(\mathrm{C}(X_{i, j}))$, and the connecting maps have the form $$f \mapsto \mathrm{diag}\{f\circ\lambda_1, ..., f\circ\lambda_n\}.$$ 
Since the connecting maps preserve the diagonal subalgebras $D_i \subseteq A_i$, $i=1, 2, ...$, their inductive limit $D$ is a commutative subalgebra of $A$. Let us assume that $A$ has no finite-dimensional quotient. Then the matrix ranks $n_{i, j}$ are arbitrarily large as $i\to\infty$, and a standard argument shows that $A$ has the property that, for any $n \in \mathbb N$, any finite set $\mathcal F \subseteq A$, and any $\eps>0$, there are mutually orthogonal projections $p_1, ..., p_n \in D$ such that $p_1+\cdots + p_n = 1$, 
$$ \abs{\tau(p_i) - \frac{1}{n}} < \eps, \quad\mathrm{and}\quad \norm{[p_i, a]}_{2, \tau} < \eps,\quad i=1, ..., n,\ a\in\mathcal F,\ \tau \in\mathrm{T}(A). $$
In other words, the conditions of Theorem \ref{Bauer-case} are satisfied for the pair $(D, A)$.
 
\begin{thm}\label{Bauer-SBP-AH}
Let $A$ be a unital separable AH algebra with diagonal maps, and assume that $A$ has no finite-dimensional quotient. If $\tr(A)$ is a Bauer simplex, then $(D, \tr(A)|_D)$ has the (SBP). In particular, by Proposition \ref{AH-Z} below, the C*-algebra $A$ is $\mathcal Z$-absorbing if it is simple.
\end{thm}

\begin{proof}
Note that (as $A$ is AH) $\tr(A)$ is canonically isomorphic to $\tr(A)|_D$. Since $\tr(A)$ is a Bauer simplex, by Theorem \ref{Bauer-case-A-free}, $(D, \tr(A)|_D)$ is approximately divisible, and so by Theorem \ref{Gamma2SBP}, $(D, \tr(A)|_D)$ has the (SBP).
\end{proof}

Just like the transformation group C*-algebra $\mathrm{C}(X)\rtimes\Gamma$ (with the (URP) and (COS)) (see \cite{EN-MD0} for $\Int$-actions, \cite{Niu-MD-Z} and \cite{Niu-MD-Z-absorbing} for the general case), an AH algebra with diagonal maps which satisfies the (SBP) is also $\mathcal Z$-absorbing:
\begin{prop}\label{AH-Z}
Let $A = \varinjlim A_n$ be a simple unital separable AH algebra with diagonal maps. Denote by $D\subseteq A$ the standard diagonal sub-C*-algebra. If $(D, \mathrm{T}(A)|_D)$ has the (SBP), then $A$ locally has slow dimension growth. That is, for any finite set $\mathcal F \subseteq A$ and any $\eps>0$, there is a unital sub-C*-algebra $C \subseteq A$ such that $\mathcal F \subseteq_\eps C$, $C \cong \bigoplus_{j=1}^d\mathrm{M}_{m_j}(\mathrm{C}(Y_j))$, and $$\frac{\mathrm{dim}(Y_j)}{m_j} < \eps,\quad j=1, ..., d.$$
In particular, $A$ is $\mathcal Z$-absorbing.
\end{prop}

\begin{proof}
The proof is similar to the proof of Theorem 4.5 of \cite{EN-MD0}. Let $\eps>0$ be arbitrary, and let $\mathcal F \subseteq  A$ be a finite subset. Without loss of generality, one may assume that $\mathcal F = \mathcal F_0 \cup \mathcal F_1$ with $\mathcal F_0 \subseteq D\cong \mathrm{C}(X)$ and $\mathcal F_1$ consisting  of finitely many matrix units of a building block. On passing to a finite open cover of $X$ and using the (SBP), the argument of the proof of Lemma \ref{unity-sbp} shows that one may assume further that $$\mathcal F_0 = \{\phi_1, ..., \phi_L\},$$
where $\phi_1, ..., \phi_L$ is a partition of unity of $X$ with the property
$$\mu_\tau(\phi_i^{-1}(0, 1)) < \frac{\eps}{L}, \quad i=1, ..., L,\ \tau \in \mathrm{T}(A).$$
Then there is a positive contraction $g\in \mathrm{C}(X)$ such that 
\begin{equation}\label{small-ring}
\tau(g) < \eps \quad \mathrm{and} \quad g \chi_{\eps/4, 1-\eps/4}(\phi_i) = \chi_{\eps/4, 1-\eps/4}(\phi_i),\quad i=1, ..., L,\ \tau \in \mathrm{T}(A),
\end{equation}
where $\chi_{\eps/4, 1-\eps/4}$ is a positive function vanishing outside $(\eps/4, 1-\eps/4)$, but is $1$ on $[\eps/2, 1-\eps/2]$.

With a sufficiently large $k$, there are positive contractions $g' \in D_k$ and $\phi_i \in D_k$, $i=1, ..., L$, such that
\begin{equation}\label{small-ring-local}
\norm{g - g'}_\infty < \eps/4 \quad \mathrm{and} \quad \norm{\phi'_i - \phi_i}_\infty < \eps/4,\quad i=1, ..., L.
\end{equation}
Without loss of generality, one may assume that
\begin{equation}\label{small-g'}
\tau(g') < \eps,\quad \tau\in\mathrm{T}(A_k),
\end{equation}
and
$$\mathcal F_1 \subseteq A_k.$$
Then, by \eqref{small-ring}, \eqref{small-ring-local}, and the fact that $\phi_i, \phi_i', g, g'$, $i=1, ..., L$, commute (they all belong to $D$), one has 
\begin{equation}\label{local-1}
\chi_{\eps, 1-\eps}((\phi'_i-\eps)_+)h_{1-\eps/4}(g') = \chi_{\eps, 1-\eps}((\phi'_i-\eps)_+), \quad i=1, ..., L,
\end{equation}
where $h_{1-\eps/4}$ is the continuous function which is $1$ on $1-\eps/4$, $0$ at $0$, and linear in between.
Also note that
\begin{equation}\label{pertu-g'}
\norm{h_{1-\eps/4}(g') - g'}_\infty < \eps/4.
\end{equation}

Then consider 
$$\theta_{\eps, 1-\eps}((\phi_i'-\eps)_+),\quad i=1, ..., L,$$
where $\theta$ is the positive function which is $0$ on $(0, \eps)$, $1$ on $(1-\eps, \infty)$, and linear in between.
By \eqref{local-1}, \eqref{small-g'}, and \eqref{pertu-g'}, 
\begin{equation}\label{pert-small-bd}
\mu_\tau(\bigcup_{i=1}^L(\theta_{\eps, 1-\eps}((\phi_i'-\eps)_+))^{-1}(0, 1)) < \tau(h_{1-\eps/4}(g')) < 2\eps,\quad \tau \in \mathrm{T}(A_k).
\end{equation}

Now, write $A_k \cong \bigoplus_{j=1}^d \mathrm{M}_{n_j}(\mathrm{C}(Z_j))$. For each $j=1, ..., d$, denote by $V_j$ the  set of standard matrix units of $\mathrm{M}_{n_j}(\Comp) \subseteq A_k$. Consider the C*-algebra
$$C:=\textrm{C*}\{ \theta_{\eps, 1-\eps}((\phi_i'-\eps)_+)), V_j: i=1, ..., L,\ j=1, ..., d\} \subseteq A_k.$$ Note that $\mathcal F_0 \subseteq_{3\eps} C$.

Then, by \eqref{pert-small-bd}, an argument similar to the proof of Theorem 4.5 of \cite{EN-MD0} shows that
$$C \cong \bigoplus_{j=1}^d \mathrm{M}_{n_j}(\mathrm{C}(Z'_j))$$
where $Z_j'$, $j=1, ..., d$ is a metrizable compact space with $\mathrm{dim}(Z'_j) \leq (2\eps) n_j$. Since $\mathcal F_0 \subseteq_{3\eps} C$ and $\mathcal F_1 \subseteq A_k$, one has that $\mathcal F \subseteq_{3\eps} C$. Since $\mathcal F$ and $\eps$ are arbitrary, this implies that $A$ locally has slow dimension growth, as desired.
\end{proof}

\bibliographystyle{plainurl}

\end{document}